\documentclass[10pt,a4paper]{article}
\usepackage{amsfonts}
\usepackage{amssymb,amsthm, amsmath,graphicx,enumerate,mathtools}
\usepackage{bbm}

\usepackage{algorithm}
\usepackage[noend]{algpseudocode}

\usepackage{multirow}

\usepackage{color}
\usepackage{hyperref}

\usepackage{url}

\newtheorem{theorem}{Theorem}
\newtheorem{definition}[theorem]{Definition}
\newtheorem{proposition}[theorem]{Proposition}
\newtheorem{corollary}[theorem]{Corollary}
\newtheorem{lemma}[theorem]{Lemma}

\theoremstyle{remark}

\newtheorem{example}[theorem]{Example}
\newtheorem{remark}[theorem]{Remark}

\def\C{\mathcal{C}}

\def\CaP{\mathbf{P}}

\def\FraC{\mathcal{C}}

\def\FraP{\mathcal{P}}
\def\LL{\mathbb{L}}
\def\GG{\mathbb{G}}
\def\VV{\mathbb{V}}
\def\k{\mathbbmss{k}}
\def\N{\mathbb{N}}
\def\R{\mathbb{R}}

\def\Q{\mathbb{Q}}

\def\Tint{\mathrm{T\mbox{-}int} }
\def\int{\mathfrak{int} }

\def\ap{\mathrm{Ap} }

\renewcommand{\H}{\mathcal H}

\title{A characterization of some families of Cohen--Macaulay, Gorenstein and/or Buchsbaum rings}

\author{J. I. Garc\'{\i}a-Garc\'{\i}a
 \footnote{
     Dpto. de Matem\'aticas/INDESS (Instituto Universitario para el Desarrollo Social Sostenible),
     Universidad de C\'adiz, E-11510 Puerto Real  (C\'{a}diz, Spain).
     E-mail: ignacio.garcia@uca.es. Partially supported by
     MTM2014-55367-P and
     Junta de Andaluc\'{\i}a group FQM-366. }\\
D. Mar\'{\i}n-Arag\'on
 \footnote{
     E-mail: daniel.marinaragon@alum.uca.es.
     Partially supported by Junta de Andaluc\'{\i}a group FQM-366.}
     \\
 A. Vigneron-Tenorio
 \footnote{
 	Dpto. de Matem\'aticas/INDESS (Instituto Universitario para el Desarrollo Social Sostenible), Universidad de C\'adiz,
 	E-11406 Jerez de la Frontera (C\'{a}diz, Spain).
     E-mail: alberto.vigneron@uca.es.
     Partially supported by MTM2015-65764-C3-1-P (MINECO/FEDER, UE) and
     Junta de Andaluc\'{\i}a group FQM-366.}
}

\date{}

\begin{document}

\maketitle

\begin{abstract}

We provide algorithmic methods to check the Cohen--Macaulayness, Buchsbaumness and/or Gorensteiness of some families of semigroup rings that are constructed from the dilation of  bounded convex polyhedrons of $\R^3_{\geq}$.
Some families of semigroup rings are given satifying these properties.


\smallskip
{\small \emph{MSC-class:} 13H10 (Primary), 20M14, 52B10,  (Secondary).}
\end{abstract}

\section*{Introduction}

A semigroup is a pair $(S,+)$ with $S$ a nonempty set and $+$ a binary operation defined on $S$ verifying the associative and commutative laws.
In addition, if there exists in $S$ an element, usually denoted by $0$, such that $a+0=a$ for all $a\in S$, the pair $(S,+)$ is a monoid.
Given a subset $A$ of a monoid $S$, the monoid generated by $A$, denoted by $\langle A\rangle$, is the least (with respect to inclusion) submonoid of $S$ containing $A$.
When $S=\langle A\rangle $, it is said that $S$ is generated by $A$, or that $A$ is a system of generators of $S$. The monoid $S$ is finitely generated if it has a finite system of generators.
Finitely generated submonoids of $\N^n$ are known as affine semigroups.

Given a semigroup $S$ and a field $\k$, define the semigroup ring
$\k[S]=\oplus_{s\in S} \k  y_s$ with the addition  performed component-wise and the
product carried out according to the rule $y_sy_{s'}=y_{s+s'}$.
If $S$ is minimally generated by $\{s_1,\dots,s_r\}\subset \N^n,$ 
then $\k[S]$ is isomorphic to the subalgebra of the polynomial ring $\k[x_1,\ldots ,x_{n}]$ generated by
$x^{s_1},\dots,x^{s_r}$ with $x^\alpha=x_1^{\alpha_1}\cdots x_n^{\alpha_n}$ where $\alpha=(\alpha_1,\dots,\alpha_n)\in\N^n$ (see \cite{TrungHoa}).
It is well known that many properties of the ring $\k[S]$ are determined by certain conditions and properties of $S.$ In particular, a semigroup $S$ is called a Cohen--Macaulay (resp. Gorenstein, resp. Buchsbaum) semigroup if $\k[S]$ is a Cohen--Macaulay (resp. Gorenstein, resp. Buchsbaum) ring.

Although the Cohen--Macaulayness, Gorensteiness, and Buchsbaumness of rings have been widely studied, there are few methods for searching for this kind of structure (see \cite{RosalesBuchs}, \cite{MR1732040}, \cite{RosalesCM}, \cite{MR881220} and references therein).
The main goal of the present work is to provide methods to construct affine semigroups satisfying these three properties by means of convex polyhedron semigroups.

Convex polytope semigroups are obtained as follows.
Assume that  $\CaP\subset \R^n_{\geq}$ is a bounded convex polytope. Then
the set $\cup_{j\in \N} j\CaP$ is a submonoid of $\R^n_{\geq}$ and the set
$\FraP= \cup_{j\in \N} j \CaP\cap \N^n$ is a submonoid of $\N^n$  (see \cite{ACBS}).
Every submonoid of $\N^n$ that can be obtained in this way is a convex polytope semigroup, and
if $n=3$, the term polyhedron is used instead of polytope.

In order to study the Cohen--Macaulayness, Buchsbaumness, and Gorensteinness of affine simplicial convex polyhedron semigroups, we need to describe the elements in the minimal cone, which includes the semigroup, that are not in the semigroup. In particular, we give a geometric description and a decomposition of the set $L_{\R_\ge}(\CaP)\setminus \cup_{j\in\N} j\CaP.$

The results of this work are illustrated with several examples, where the third-party software,
such as
{\tt Normaliz} (\cite{normaliz}) and 
{\tt Macaulay2} (\cite{Macaulay2})
are used in conjunction with the library \cite{PROGRAMA} developed by the authors in Python (\cite{python}).

This work is organized as follows. In Section \ref{s3},
the concept of convex polytope/polyhedron semigroup is defined and
some basic notions and results used  in the rest of the work are given.
In Section \ref{s4}, the set $L_{\R_\ge}(\CaP)\setminus \cup_{j\in\N} j\CaP$ is completely described by using geometric tools.
In Section \ref{sectionCM},
the Cohen--Macaulay property is studied and
an algorithm is presented that checks for this property in affine simplicial convex polyhedron semigroups.
A family of Gorenstein affine semigroups is given in Section \ref{gorenstein}. Lastly,
in Section \ref{buchsbaum},  Buchsbaum affine simplicial convex polyhedron semigroups are characterized and a family of such semigroups is obtained.

\section{Some definitions and tools}\label{s3}

Let $\R$, $\Q$ and $\N$ be the sets of real numbers, rational numbers, and nonnegative integers, respectively.
Denote by $\R_\geq$ and $\Q_\geq$ the set of nonnegative elements of $\R$ and $\Q$.
A point $P$ is said to be rational whenever $P\in\Q^n$. We denote by $O$ the origin of $\R^n.$
A rational line is a straight line
with at least two rational elements.
A ray is a vector line, and whenever it is also a rational line, it is called a rational ray.

Define the cone generated by $A\subseteq \R^n_{\geq}$ by
$$
L_{\R_{\geq}}(A)=\left\{\sum_{i=1}^p q_ia_i\mid p\in\N, q_i\in \R_{\geq}, a_i\in A \right\}.
$$
The set of its extremal rays is denoted by $\mathcal{T}_A$, and
the set $L_{\R_{\geq}}(A)\cap \N^n$ is denoted by $\FraC_A$. A cone is called a rational cone if all its extremal rays are rational rays.

For every set $\chi\subset \R_\ge^n$,
denote
by $\partial \chi$ its topological boundary,
by $\Tint(\chi)$ its topological interior, and by $\int(\chi)$ the set $ \chi\setminus \mathcal{T_\chi}.$
For any $i$ nonnegative integer, we use $[i]$ to denote the set $\{1,\ldots ,i\}$.

If $P=(p_1,\dots,p_n), Q=(q_1,\dots,q_n)\in\R^n$ and $j\in \R$, then $jP=(jp_1,\dots,jp_n)$ and $P+Q=(p_1+q_1,\dots,p_n+q_n)$. For every subset $A\subset \R^n,$  $jA$ denotes the set $\{jP\mid P\in A\},$
and $P+A$ (or $A+P$) denotes the set $\{P+Q\mid Q\in A\}.$
Let $\H(A)$ be the convex hull determined by $A$, that is, the smallest convex set containing $A.$
In general, for every $P\in\Q_{\geq}^n$, define $h_P=\min\{h\in \N | h P\in \N^n\}.$
For a given $P\in\R^n,$ we denote by $\tau_P$ the ray containing $P.$

The main characteristic of convex polytope semigroups is that in order to know whether a given element $P$ of $\N^n$ is in the semigroup, it is enough to know the defining equations of the facets of the convex polytope $\CaP$.
From these equations, by computing the sets $\tau_P\cap \CaP=\overline{AB}$ and
$\{k\in\N|\frac{||P||}{||B||}\leq k \leq \frac{||P||}{||A||}\}$
(with $||X||$ the Euclidean norm of $X,$ that is, $||X||^2=||(x_1,\dots,x_n)||^2=\sum_{j=1}^n x_j^2$),
we get that the element $P$ belongs to $\FraP$ if and only if
$\{k\in\N|\frac{||P||}{||B||}\leq k \leq \frac{||P||}{||A||}\}\neq \emptyset$.
This process can be performed even when $\FraP$ is not finitely generated.
Its well known that $\FraP$ is a finitely generated semigroup if and only if $\tau\cap\CaP\cap\Q^n$ is non-empty for all $\tau\in \mathcal{T}_\CaP$ (see \cite[Corollary 2.10]{Brunz}).

We now recall some definitions that will be necessary for the comprehension of Sections 3, 4, and 5. If $R$ is a Noetherian local ring, we say that a finite $R$-module $M\neq 0$ is Cohen--Macaulay whenever ${\rm depth}(M)=\dim(M)$. The set $R$ is a Cohen--Macaulay ring if it is a Cohen--Macaulay module (see \cite{MR1251956}). If $R$ is a local ring with finite injective dimension as $R$-module, then it is called a Gorenstein ring. Finally, we define a Buchsbaum ring as a Noetherian $R$-module such that every system of parameters of $M$ is a weak $M$-sequence and a Buchsbaum ring as a Buchsbaum module, that is, a module over itself (see \cite{MR0153708}).

\section{Computing a decomposition of $L_{\R_\ge}(\CaP)\setminus \cup_{j\in\N}j\CaP$
}\label{s4}

Let $\CaP\subset \R_\ge ^3$ be a convex polyhedron and $\FraP$ be its associated semigroup.
In this section, we give a geometric description of the set $L_{\R_\ge}(\CaP)\setminus \cup_{j\in\N}j\CaP;$ this description allows us to give a decomposition of this set as a union of convex polyhedra.
For example, Figure \ref{ejemplo_figura} illustrates the sets $L_{\R_\ge}(\CaP)$ and $\cup_{j=0}^9j\CaP$ for a polyhedron $\CaP.$
Two vertices of $\CaP$ are called adjacent if they are connected by an edge.
\begin{figure}[h]
\begin{center}
\begin{tabular}{|c|}\hline
\includegraphics[scale=.25]{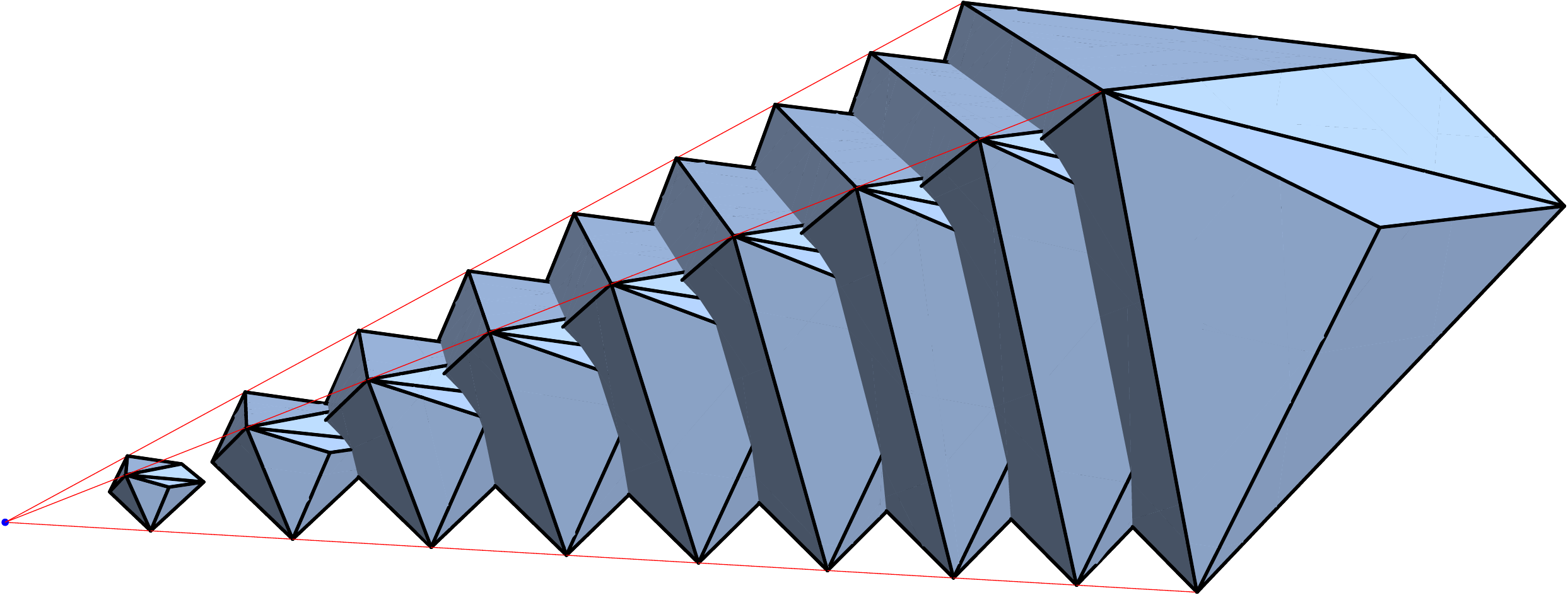}\\
\hline
\end{tabular}
\caption{Example of $L_{\R_\ge}(\CaP)$ and $\cup_{j=0}^9j\CaP.$}\label{ejemplo_figura}
\end{center}
\end{figure}

For any $j\in \N,$ we denote by $\LL_j=\H(j\CaP\cup(j+1)\CaP)\setminus (j\CaP\cup(j+1)\CaP),$ by $\overline{\LL}_j$ its topological closure, and by $\LL$ and $\overline{\LL}$ the sets $\cup_{j\in \N} \LL_j$ and $\cup_{j\in \N} \overline{\LL}_j$, respectively. Note that $\LL\cap \N^3=\FraC_\FraP\setminus \FraP.$
In general, the sets $\overline{\LL}_j$ are not convex, althought the can be expressed as a union of convex polyhedron (see (\ref{igualdad})). This allow us to use the concepts of vertex, edge and face of these sets.

We write $\mathcal{T}_\CaP=\{\tau_1,\ldots ,\tau_t\}$ for the set of extremal rays of $L_{\R_\ge}(\CaP).$ The symbol $\sigma$ denotes the permutation $(123\cdots t),$ that is, for $i\in[t-1],$ $\sigma(i)=i+1,$ and $\sigma(t)=1,$ and we assume that the rays in $\mathcal{T}_\CaP$ are arranged in such a way that $\tau_i$ and $\tau_{\sigma(i)}$ are in the same face of $L_{\R_\geq}(\CaP)$ for every $i\in[t].$

If $\mathcal{V}$ is the set of vertices of $\CaP$, $P\in \mathcal{V}$, and $\tau_P$ is the ray containing $P$, then there exists a unique $P'\in\CaP$ (which can be equal to $P$) such that
$\tau_P\cap \CaP$ is equal to the segment $\overline{PP'}.$
So, $\mathcal{V}$ is equal to $\mathcal{W}\sqcup  \mathcal{W}_1\sqcup \mathcal{W}_2\sqcup\mathcal{V}_1\sqcup \mathcal{V}_2$ with
$\mathcal{W}=\{P\in \mathcal{V}\mid P=P'\}$,
$\mathcal{W}_1=\{P\in \mathcal{V}\cap\cup_{i\in[t]}\tau_i \mid ||P||<||P'|| \}$,
$\mathcal{W}_2=\{P\in \mathcal{V}\cap\cup_{i\in[t]}\tau_i \mid ||P'||<||P||\}$,
$\mathcal{V}_1=\{ P\in \mathcal{V}\setminus(\cup_{i\in[t]}\tau_i)\mid ||P||<||P'|| \},$ and
$\mathcal{V}_2=\{  P\in \mathcal{V}\setminus(\cup_{i\in[t]}\tau_i)\mid ||P'||< ||P|| \}.$
It is straightforward to prove that for every $P\in \mathcal{V}:$
\begin{itemize}
\item $P\in \mathcal{W}$ iff $\tau_P\in\mathcal{T}_\CaP$ and $\tau_P\cap \CaP=\{P\};$
\item $P\in \mathcal{W}_1$ iff $\tau_P\in\mathcal{T}_\CaP$ and $\tau_P\cap \CaP=\overline{PP'}$ with $||P||<||P'||;$
\item $P\in \mathcal{W}_2$ iff $\tau_P\in\mathcal{T}_\CaP$ and $\tau_P\cap \CaP=\overline{P'P}$ with $||P'||<||P||;$
\item $P\in \mathcal{V}_1$ iff $\tau_P\notin\mathcal{T}_\CaP$ and $\tau_P\cap \CaP=\overline{PP'}$ with $||P||<||P'||;$
\item $P\in \mathcal{V}_2$ iff $\tau_P\notin\mathcal{T}_\CaP$ and $\tau_P\cap \CaP=\overline{P'P}$ with $||P'||<||P||.$
\end{itemize}

We add new points to the sets $\mathcal{V}_1$ and $\mathcal{V}_2$.
For every $P,Q\in \mathcal W$ satisfying that if $P\in\tau_i$ then $Q\not \in\tau_{\sigma(i)}\cup \tau_{\sigma(i-1)}$, consider the segment $\overline {PQ}$. The set $\overline {PQ}\setminus\{P,Q\}$ is included in  $\Tint( L_{\R_{\geq}}(\CaP) )$. Thus, for every $A \in \overline {PQ}\setminus\{P,Q\}$ the set $\tau_A\cap \CaP$ is a nonempty segment and there exists $k\in\N\setminus\{0\}$ such that $k\overline {PQ} \cap (k+1) \CaP \neq \emptyset$ or  $k\overline {PQ} \cap (k-1) \CaP \neq \emptyset$. We consider the minimum of such $k$:
if $k\overline {PQ} \cap (k-1) \CaP\neq \emptyset$, take $B\in k\overline {PQ} \cap (k-1) \CaP$ and add $\frac 1 k B$ to $\mathcal{V}_1$; otherwise, add $\frac 1 k B$ to $\mathcal{V}_2$.

Note that for every $A\in\mathcal{V}_1\cup \mathcal{W}_1$ the set $\overline{OA}\cap \CaP=\{A\}$ and for every $A\in\mathcal{V}_2\cup \mathcal{W}_2$ the set $\overline{OA}\cap \CaP$ is a segment. We can say that the points of $\mathcal{V}_2\cup \mathcal{W}_2$ are hidden for an observer located at the origin.

From now on, we denote by $P_i$ the point in $(\mathcal{W}\cup \mathcal{W}_1)\cap \tau_i$ for each $i\in [t].$ The next result describes $\LL$ when $\tau_i\cap \CaP$ is a segment for every $i\in[t].$

\begin{lemma}\label{Csemig}
Let $\CaP\subset \R^3_\ge$ be a convex polyhedron such that $\mathcal{W}=\emptyset.$ Then there exists $k\in\N$ satisfying $\LL=\cup_{j=0}^{k-1} \LL_j\subset \H(\{O,kP_1,\ldots,kP_t\}).$
\end{lemma}

\begin{proof}
For any $P_i\in\mathcal{W}_1,$ denote by $\overline{P_iP'_i}$ the segment $\tau_i\cap \CaP.$ Note that for any $j\ge 0$ and $i\in[t],$ $\tau_i\cap j\CaP$ is equal to $j\overline{P_iP'_i},$ and there exists a minimal $\kappa_0\in \N$ such that $k\overline{P_iP'_i}\cap (k+1)\overline{P_iP'_i}\neq \emptyset$ for every $i\in[t]$ and $k\ge \kappa_0.$ So $\cup _{j\ge \kappa_0}j\CaP$ is a convex set, and $L_{\R_\ge}(\CaP)\setminus \cup_{j\in\N}j\CaP\subset \cup_{j=0}^{\kappa_0-1} \LL_j.$
Hence, the set $\LL$ is equal to $\cup_{j=0}^{\kappa_0-1} \LL_j,$ and it is included in $\H(\{O,\kappa_0P_1,\ldots,\kappa_0P_t\}).$
\end{proof}

The previous lemma can be generalized easily for any dimension. A convex polytope semigroup $\FraP$ such that $\FraC_\FraP\setminus\FraP$ is finite is called a $\C_\FraP$-semigroup. Some properties of this kind of semigroup have been studied recently in \cite{Csemigroup}.

In the following, we assume that there exists at least one ray $\tau_i$ in $\mathcal{T}_\CaP$ such that $\tau_i\cap \CaP$ is a point, that is, $\mathcal{W}\neq\emptyset.$ The next two lemmas prove some properties of the intersection of the polyhedra $k\CaP$ and $(k+1)\CaP$ for any integer $k\gg 0.$

\begin{lemma}\label{kcota}
Let $Q\in\mathcal{V}_1\cup\mathcal{W}_1$ and $Q'\in\mathcal{V}_2\cup\mathcal{W}_2.$ Then there exists  $\kappa_0\in \N$ such that for every integer $k\geq \kappa_0$, $(k+1)Q\in \Tint(k\CaP)$ and $kQ'\in\Tint((k+1)\CaP)$.
\end{lemma}

\begin{proof}
Let $\tau_Q$ be the ray containing $Q$.
There exists $R$ such that $\tau_Q\cap \CaP=\overline{QR}$ and $||Q||<||R||$.
This implies that there exists a least $k_Q\in\N$ such that $k_Q||Q||<(k_Q+1)||Q||<k_Q||R||$,
and therefore $(k_Q+1)Q\in \Tint(k_Q\CaP)$. Similarly, for every
$Q'\in\mathcal{V}_2\cup\mathcal{W}_2$, there exists a $k_{Q'}\in\N$ such that
$k_{Q'}Q'\in\Tint((k_{Q'}+1)\CaP)$. We take $\kappa_0$ to be the maximum of the set
$\{k_Q\mid Q\in \mathcal{V}_1\cup\mathcal{W}_1\}\cup\{ k_{Q'}\mid Q'\in\mathcal{V}_2\cup\mathcal{W}_2 \}$.
\end{proof}

Note that $\kappa_0$ is the first time that all the vertices in $(k+1)(\mathcal{V}_1\cup\mathcal{W}_1)$ are in $\Tint(k\CaP)$ and all the vertices in $k(\mathcal{V}_2\cup\mathcal{W}_2)$ belong to $\Tint((k+1)\CaP)$ for every integer $k$.

\begin{lemma}\label{lemak0}
For every integer $k\geq \kappa_0$, every vertex of $\overline{\LL}_k$ belongs to $k\mathcal{W}\cup (k+1)\mathcal{W},$ or to an edge of
$(k+1)\overline{PQ}$ of $(k+1)\CaP,$ or is contained in an edge $k\overline{PQ'}$ of $k\CaP,$ with $P\in \mathcal{W},$ $Q\in \mathcal{V}_1\cup\mathcal{W}_1$ and $Q'\in \mathcal{V}_2\cup\mathcal{W}_2.$

\end{lemma}

\begin{proof}
Since $\LL_k=\H(k\CaP\cup(k+1)\CaP)\setminus (k\CaP\cup(k+1)\CaP)$,
every vertex of $\overline{\LL}_k $ belongs either to an edge of $k\CaP$ or to an edge of $(k+1)\CaP$.

Clearly, if $\overline{Q_1Q_2}$ is an edge of $\CaP$ with $Q_1,Q_2\in\mathcal{V}_1\cup\mathcal{W}_1$ and $k\geq \kappa_0$,
then $(k+1)\overline{Q_1Q_2}\subset \Tint({k\CaP})$
and therefore  it does not
contain vertices of $\overline{\LL}_k $.
We can perform a similar reasoning in case
$Q_1,Q_2\in\mathcal{V}_2\cup \mathcal{W}_2$ to obtain that $\overline{Q_1Q_2}$ does not
contain vertices of $\overline{\LL}_k $ for every $k\geq \kappa_0.$ So, the vertices of $\overline{\LL}_k$ belong to $(k+1)\overline{PQ}$ or $k\overline{PQ'}$ for every $Q\in \mathcal{V}_1\cup\mathcal{W}_1$ and $Q'\in \mathcal{V}_2\cup\mathcal{W}_2.$
\end{proof}

The previous result means that for every integer $k\ge \kappa_0,$ each vertex in $\overline{\LL}_k$ is adjacent to at least one vertex in $k\mathcal{W}\cup (k+1)\mathcal{W},$ and they can be determined. For $P$ any point in $\mathcal{W},$ we define the sets $\VV_{1P}=\{Q\in\mathcal{V}_1\cup \mathcal{W}_1\text{ adjacent to }P\},$ $\VV_{2P}=\{Q'\in\mathcal{V}_2\cup \mathcal{W}_2\text{ adjacent to }P\},$ and
$$\GG_P^k=\{kP,(k+1)P\}\bigcup \cup_{Q\in\VV_{1P}}((k+1)\overline{PQ}\cap k\partial\CaP) \bigcup \cup_{Q'\in\VV_{2P}}(k\overline{PQ'}\cap (k+1)\partial\CaP).$$

\begin{theorem}
The set $\cup_{P\in\mathcal{W}}\GG_P^k$ is the vertex set of $\overline{\LL}_k$ for all integers $k\ge \kappa_0.$
\end{theorem}

\begin{proof}
From Lemmas \ref{kcota} and \ref{lemak0}.
\end{proof}

\begin{corollary}\label{traslacion}
Let $P$ be an element in $\mathcal{W}.$ For every $j\in \N,$ $\GG_P^{\kappa_0+j}=\GG_P^{\kappa_0}+jP.$
\end{corollary}

\begin{proof}
By definition, the intersection $\tau_P\cap \GG_P^{\kappa_0+j}$ is $\{(\kappa_0+j)P,(\kappa_0+j+1)P\},$ that is, it is equal to $(\tau_P\cap \GG_P^{\kappa_0})+jP.$

Assume $Q\in\VV_{1P},$ and let $\pi$ be the plane containing $\{O,P,Q\},$ let $\CaP'$ be the convex polygon $\CaP\cap \pi,$ and let $Q'\in\CaP'\setminus\{Q\}$ be the other vertex of $\CaP'$ adjacent to $P.$
After some basic computations the reader can check that for every real number $m\gg 0,$ $((m+2)\overline{PQ}\cap (m+1)\overline{PQ'})-((m+1)\overline{PQ}\cap m\overline{PQ'})=P.$
Hence, $P'\in\cup_{Q\in\VV_{1P}}((\kappa_0+j+1)\overline{PQ}\cap (\kappa_0+j)\partial\CaP)$ iff $P'\in jP+\cup_{Q\in\VV_{1P}}((\kappa_0+1)\overline{PQ}\cap \kappa_0\partial\CaP).$
Analogously, a similar result can be proved for any $Q\in \VV_{2P}.$
\end{proof}

So far, we have described the vertex set of $\overline{\LL}_k$ for any integer $k\ge \kappa_0.$ Now we are going to construct a decomposition of $\overline{\LL}_k.$ First of all, we need to rearrange the points in $\GG_{P_i}^{k}\setminus\tau_i$ for all $P_i\in \mathcal{W}$ (recall that we have denoted by $P_i$ the point in $(\mathcal{W}\cup \mathcal{W}_1)\cap \tau_i$ where $i\in [t]$). Let $\{Q_1,\dots,Q_m\}\subset k\CaP\cap (k+1)\CaP$ be the set $\GG_{P_i}^{k}\setminus\tau_i.$
Thus, for every $j\in[m]$, the segments $\overline{(kP)Q_j}$ and $\overline{((k+1)P)Q_j}$ are edges of $\overline{\LL}_k.$ We can rearrange these elements so that for every $j\in[m-1]$, the elements $Q_j$ and $Q_{j+1}$ are in the same face of $k\CaP$ and in the same face of $(k+1)\CaP.$
In this way, we obtain that the segments $\overline{Q_1Q_2},\dots,\overline{Q_{m-1}Q_m}$ are also edges of
$\overline {\LL}_k $ and the triangles $(kP)Q_jQ_{j+1}$ and $((k+1)P)Q_jQ_{j+1}$ are faces of $\overline {\LL}_k.$ In the following, we assume that the elements in $\GG_{P_i}^{k}\setminus\tau_i$ are arranged in this way.

The last vertex $Q_m\in \GG_{P_i}^{k}\setminus\tau_i$ satisfies there exists two points $A,B$ in $k\CaP\cap \tau_{\sigma(i)}$ and $(k+1)\CaP\cap \tau_{\sigma(i)}$ respectively, or in $k\CaP\cap\tau_{\sigma(i-2)}$ and $(k+1)\CaP\cap \tau_{\sigma(i-2)}$ such that $\overline{Q_mA}\subset k\partial \CaP$ and $\overline{Q_mB}\subset (k+1)\partial \CaP.$
For the vertex $Q_1$ we have a similar situation.
In the sequel, we assume that $Q_m$ and $\tau_{\sigma(i)}$ satisfy the above property as well as the pair $Q_1$, $\tau_{\sigma(i-2)}$.
If $P_{\sigma(i)}\in\mathcal{W},$ let $\{Q'_1,\dots,Q'_{m'}\}= \GG^k_{P_{\sigma(i)}}\setminus \tau_{\sigma(i)}.$ Since the segment $k\overline{P_iP_{\sigma(i)}}$ is an edge of $\overline{\LL}_k,$ $Q_m$ and $Q'_1$ belong to a face of $k\CaP$ containing $kP_{\sigma(i)}$ and $kP_i$ respectively. This situation is illustrated in Figure \ref{ejemplo_cortes}.

\begin{definition}\label{rearrange}
Let $P_i\in\mathcal{W}.$ Denote by $T^{k}_{ij}$ the tetrahedron with vertices $kP_i$, $(k+1)P_i$, $Q_{j}$, and $Q_{j+1}$ with $j\in[m-1],$ and by $G_i^{k}$ the set $\cup_{j\in[m-1]} T^{k}_{ij}.$ If $P_i\in \mathcal{W}_1$ or $m=1,$ we fix $G_i^{k}=\emptyset.$
If $P_{i}\in\mathcal{W}_1$ or $P_{\sigma(i)}\in\mathcal{W}_1$, consider $\widetilde{G}_{i\sigma(i)}^{k}=\emptyset,$ otherwise, we define $\widetilde{G}_{i\sigma(i)}^{k}$ as the convex hull of the triangles $(kP_i)((k+1)P_i)Q_{m}$ and   $(kP_{\sigma(i)})((k+1)P_{\sigma(i)})Q'_1.$
\end{definition}

\begin{remark}\label{rearrange1}
Note that the set $G^{k}_i$ is the set determined by $\GG^{k}_{P_i}$ when $\GG^{k}_{P_i}$ is not a 2-dimensional object. So, for all $q\in \N,$ $G_i^{\kappa_0+q}=G_i^{\kappa_0}+qP_i$ and $T^{\kappa_0+q}_{ij}=T^{\kappa_0}_{ij}+qP_i$ (Corollary \ref{traslacion}).
If $\widetilde{G}_{i\sigma(i)}^{k}\neq \emptyset$, and since the edge $\overline{Q_{m}Q'_{1}}$ is the intersection of a face that contains $k\overline{P_iP_{\sigma(i)}}$ with another containing $(k+1)\overline{P_iP_{\sigma(i)}},$ $\overline{Q_{m}Q'_{1}}$ is also parallel to $k\overline{P_iP_{\sigma(i)}}$ and $(k+1)\overline{P_iP_{\sigma(i)}}.$ Moreover, $\widetilde{G}_{i\sigma(i)}^{\kappa_0+j}$ is the convex hull of the triangles $(\kappa_0P_i)((\kappa_0+1)P_i)Q_{m}+jP_i$ and $(\kappa_0P_{\sigma(i)})((\kappa_0+1)P_{\sigma(i)})Q'_1+jP_{\sigma(i)}$ for all $j\in \N$ (Corollary \ref{traslacion} again).
\end{remark}

The above sets provide us a decomposition of the sets $\overline{\LL}_k$:
\begin{equation}\label{igualdad}
\overline{\LL}_k=\bigcup_{i\in[t]} (G_i^{k} \cup \widetilde{G}_{i\sigma(i)}^{k}),
\textrm{ for every integer } k\ge \kappa_0.
\end{equation}
Summarizing, the set $\LL$ is equal to
$$\left(\H(\{O,\kappa_0P_1,\ldots ,\kappa_0P_t\})\setminus\cup_{k=0}^{\kappa_0-1}k\CaP\right)\cup \bigcup_{k\ge \kappa_0}\bigcup_{i\in[t]} \Tint(G_i^{k} \cup \widetilde{G}_{i\sigma(i)}^{k}).$$

Since $\overline{\LL}_{\kappa_0+j}$ can be obtained from $\overline{\LL}_{\kappa_0}$ by adding $jP_i$ to $G_i^{\kappa_0},$ and $jP_i$ or $jP_{\sigma(i)}$ to the vertices of $\widetilde{G}_{i\sigma(i)}^{\kappa_0},$ $\overline{\LL}_{\kappa_0+j}$ can be determined from a finite number of subsets.
Note that all the elements necessary to delimit $\LL$ can be  easily computed by using basic geometric tools.
In Figure \ref{ejemplo_cortes}, you can see an example of the sets $\LL_k,$ and the vertex sets of $\widetilde{G}_{i\sigma(i)}^{k}$ and $G_i^{k}.$
\begin{figure}[h]
\begin{center}
\begin{tabular}{|c|}\hline
\includegraphics[scale=.25]{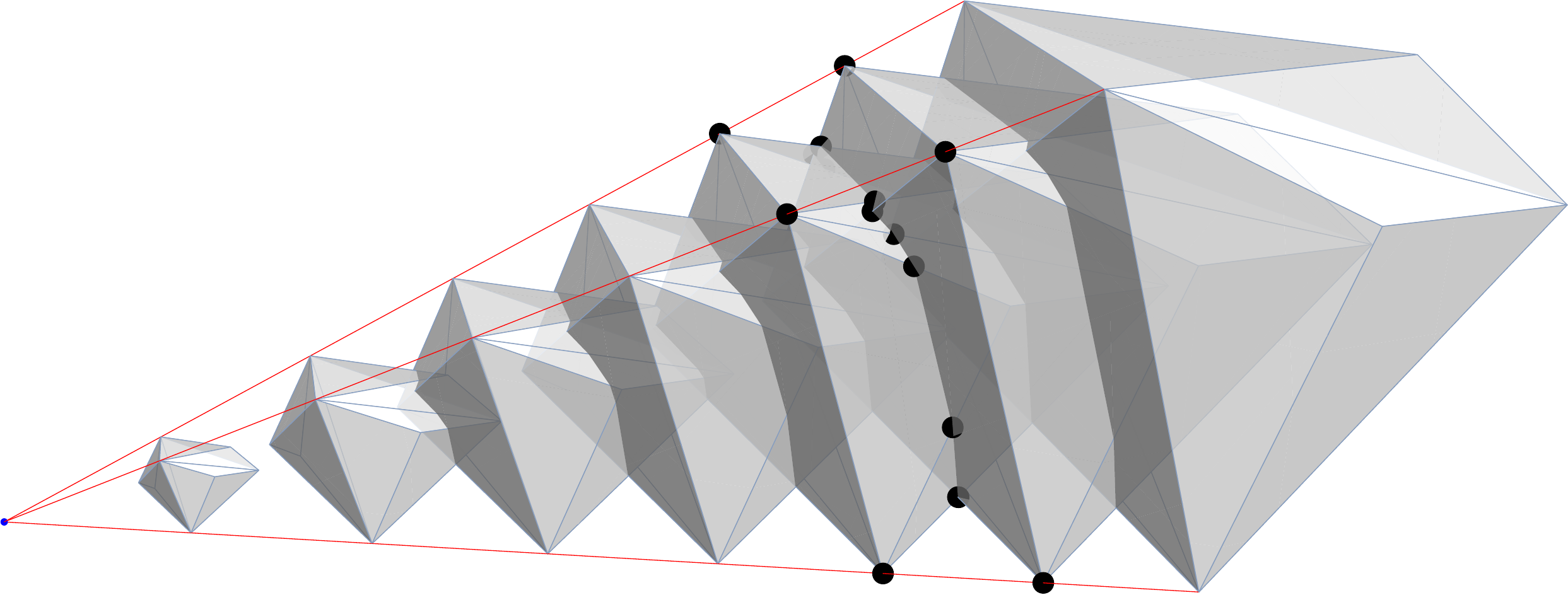}\\
\hline
\end{tabular}
\caption{Vertex set of $\LL_5.$}\label{ejemplo_cortes}
\end{center}
\end{figure}

The last results in this section are two properties of the sets $G_i^{k}$ and $\widetilde{G}_{i\sigma(i)}^{k}$ that are used in the following sections.

\begin{lemma}\label{deltaP}
Let $P$ be a point in $\mathcal{W}\cap\Q^3_\geq$.
For every integer $k\geq \kappa_0$ there exist $r,q\in\N$ with $0\leq r<h_P$  such that
$G_i^{k}\cap \N^3=(G_i^{\kappa_0+r}\cap \N^3)+qh_P P$.
\end{lemma}
\begin{proof}
Take  $q$ and $r$ in $\N$ such that $k-\kappa_0=q h_P + r$ with $0\leq r< h_P$.
\end{proof}

\begin{lemma}\label{prisma}
If $P\in \widetilde{G}_{i\sigma(i)}^{k}$ and $k\in\N$ is greater than or equal to  $\kappa_0,$ then $P+P_i$ and $P+P_{\sigma(i)}$ belong to $\widetilde{G}_{i\sigma(i)}^{k+1}.$
\end{lemma}

\begin{proof}
Recall that $\overline{Q_{m}Q'_{1}},$ $k\overline{P_iP_{\sigma(i)}}$ and $(k+1)\overline{P_iP_{\sigma(i)}}$ are parallel, and $\widetilde{G}_{i\sigma(i)}^{k+j}$ is the convex hull of the triangles $(kP_i)((k+1)P_i)Q_{m}+jP_i$ and   $(kP_{\sigma(i)})((k+1)P_{\sigma(i)})Q'_1+jP_{\sigma(i)}.$ If $P\in \widetilde{G}_{i\sigma(i)}^{k},$ there exists $Q$ in the triangle $(kP_i)((k+1)P_i)Q_m$ and $\lambda \in \R_\ge $ such that $P=Q+\lambda \overrightarrow{P_{i}P_{\sigma(i)}}.$ Thus, $P+P_i=(Q+P_i)+\lambda \overrightarrow{P_{i}P_{\sigma(i)}}$ where $Q+P_i$ belongs to the triangle $(kP_i)((k+1)P_i)Q_{m}+P_i.$ Hence, $P+P_i\in \widetilde{G}_{i\sigma(i)}^{k+1}.$ Analogously, it can be proved that $P+P_{\sigma(i)}$ belongs to $\widetilde{G}_{i\sigma(i)}^{k+1}.$
\end{proof}

\section{Cohen--Macaulay affine simplicial convex polyhedron semigroups}\label{sectionCM}

In this section, the Cohen--Macaulay property is studied.
A semigroup is said to be normal if $ms\in S$ for some $m\in \N$ implies $s\in S$. Hochster's theorem  proves that  normal semigroups are Cohen--Macaulay (see \cite{Hochster}).
In general, polyhedron semigroups are non-normal: take for instance
the polyhedron semigroup associated to the convex hull of
$\{(6,0,0),(0,6,0),(0,0,6),(2.2,2.2,2.2)\}$, the element
$(2,2,2)$ is in the semigroup, but $(1,1,1)$ is not.

Recall that we denote by $\mathcal{T}_A$
the set of the extremal rays of the set $L_{\R_{\geq}}(A)$, by $\FraC_A$
the set $L_{\R_{\geq}}(A)\cap \N^n$, and by $\int(A)$
the set $ A\setminus \mathcal{T}_A,$ where $A$ is a subset of $\R^n.$

The following results prove some characterizations of Cohen--Macaulay rings in terms of
the corresponding associated semigroup. Assume that $S\subset \N^n$ is a finitely generated semigroup
with $\{g_1,\ldots ,g_t\}$ its minimal generating set and that $S$ is a simplicial
semigroup, that is, $L_{\R_\ge}(S)$ is generated by $n$ linearly independent generators
of $S$. We assume that $L_{\R_\ge}(S)$ is generated by $\{g_1,\ldots ,g_n\}$ with
$n\leq t$. Note that the elements of $\{g_1,\ldots ,g_n\}$ belong to
different rays of $L_{\R_\ge}(S)$ and so $\{g_1,\ldots ,g_n\}$ is also a basis of
$\R^n$.

\begin{theorem}\label{C-M}
Let $S$ be a simplicial semigroup.
Then the following conditions are equivalent:
\begin{enumerate}[(1)]
\item $S$  is Cohen--Macaulay.
\item For any $\alpha,\beta\in S$, if there exist $1\le i\neq j\le n$ such that $\alpha+g_i=\beta+g_j$, then $\alpha-g_j=\beta-g_i\in S.$
\item For all $a\in\FraC_S\setminus S,$ and every two different integers $i,j\in [n]$, it is fulfilled that  $\{a+g_i,a+g_j\}\not \subset S$.
\end{enumerate}
\end{theorem}

\begin{proof}
By \cite[Theorem 2.2]{GSW},
(1) is equivalent to (2).

We prove now that (2) implies (3).
Suppose that there exists $a\in \FraC_S\setminus S$ and $i,j\in [n]$ such that
$\alpha=a+g_i\in S$ and $\beta=a+g_j\in S$. Clearly $\alpha+g_j=\beta+g_i$, but
$\alpha-g_i=\beta-g_j=a\not\in S$.

Lastly, we prove (3) implies (2). Let $\alpha,\beta$ be two elements of $S$ satisfying
$\alpha+g_1=\beta+g_2$.
Since $\alpha,\beta\in L_{\R_\geq}(S)=L_{\R_\geq}(\{g_1,\ldots ,g_n\})$ and
$\{g_1,\ldots ,g_n\}$ is a basis of $\R^n$, there exist
$\lambda_1,\ldots,\lambda_n,\mu_1,\ldots,\mu_n\in \R_{\geq}$
such that $\alpha=\lambda_1g_1+\lambda_2g_2+\cdots +\lambda_ng_n$ and
$\beta=\mu_1g_1+\mu_2g_2+\cdots +\mu_ng_n$. Furthermore, these expressions are unique.
Since $\beta=\alpha+g_1-g_2=(\lambda_1+1)g_1+(\lambda_2-1)g_2+\cdots +\lambda_ng_n$,
we obtain that
$\lambda_2-1=\mu_2\geq 0$, and hence $\lambda_2\geq 1$. Analogously, we obtain
that $\mu_1\geq 1$. Thus $a=\alpha-g_2=\beta-g_1\in L_{\R_\geq}(\{g_1,\ldots ,g_n\})$.
By the hypothesis and
since $a+g_1=\beta\in S$ and $a+g_2=\alpha\in S$, the element $a=\alpha-g_2=\beta-g_1$
belongs to $S$.
\end{proof}

\begin{corollary}\label{lemma_no_C-M}
Let $S\subset \N^n$ be a simplicial affine semigroup such that
$\int (\FraC_S) \setminus S$ is a finite set with $\FraC_S \neq S$. Then $S$ is not Cohen--Macaulay.
\end{corollary}

\begin{proof}
If we assume that $\int (\FraC_S) \setminus S$ is a nonempty finite set, there exists
$P\in \int (\FraC_S) \setminus S$ such that  $||P||\ge ||P'||$ for all
$P'\in \int (\FraC_S) \setminus S.$ It is straightforward to prove that this element
satisfies $P+g_i\in S$ for any $i\in [n]$.
If $\int (\FraC_S) \setminus S$ is empty, let $P\in \FraC_S\setminus S$ be a point
belonging to an extremal ray of $S,$ for example $\tau_1.$ Note that
$P+g_i\in\int(\FraC_S)\subset S$ for all $i\in [n]\setminus \{1\}.$
In any case, by Theorem \ref{C-M}, this implies that $S$ is not Cohen--Macaulay.
\end{proof}

We now study the three dimensional case.
Let $\CaP$ be a convex polyhedron delimited by the noncoplanar vertex set
$\{P_1,\ldots ,P_q\}$. From now on, we assume that its associated semigroup $\FraP$ is
simplicial, finitely generated, and that $\tau_1,$ $\tau_2$, and $\tau_3$ are the extremal rays of
$L_{\R\ge}(\CaP)$.
We denote by $\{g_1,g_2 ,g_3,\ldots ,g_t\}$ the minimal generating set of $\FraP$,
and assume that  $L_{\R_\geq}(\FraP)=L_{\R_\geq}(\{g_1,g_2,g_3\})$.

The construction given in Definition \ref{rearrange}, and some properties of the sets $G_i^k$, are used in the following result to characterize the Cohen--Macaulayness of some convex polyhedron semigroups.

\begin{lemma}\label{1pto_2segmentosCM}
Let $\CaP\subset \R^3_\ge$ be a convex polyhedron such that $\mathcal{W}=\emptyset$, or
$\CaP\cap \tau_1=\{P_1\}$ and $\CaP\cap \tau_i$ is a segment for $i=2,3.$ The affine semigroup
$\FraP$ is Cohen--Macaulay if and only if $\FraC_\FraP=\FraP.$
\end{lemma}

\begin{proof}
If $\mathcal{W}=\emptyset$, $\tau_i\cap \CaP$ is a segment for all $i\in [3],$
$\FraC_\FraP\setminus\FraP$ is a finite set.
We conclude that $\FraP$ is Cohen--Macaulay iff $\FraC_\FraP\setminus\FraP$ is
the empty set (see Corollary \ref{lemma_no_C-M}).

For the other case, we assume that $\CaP\cap \tau_2= \overline{P_2P'_2}$ and $\CaP\cap \tau_3= \overline{P_3P'_3},$ and let $\kappa_0\in\N$ be as in Lemma \ref{kcota} (Figure \ref{poliedro_1pto_2segment} illustrates these cases).
\begin{figure}[h]
\begin{center}
\begin{tabular}{|c|}\hline
\includegraphics[scale=.22]{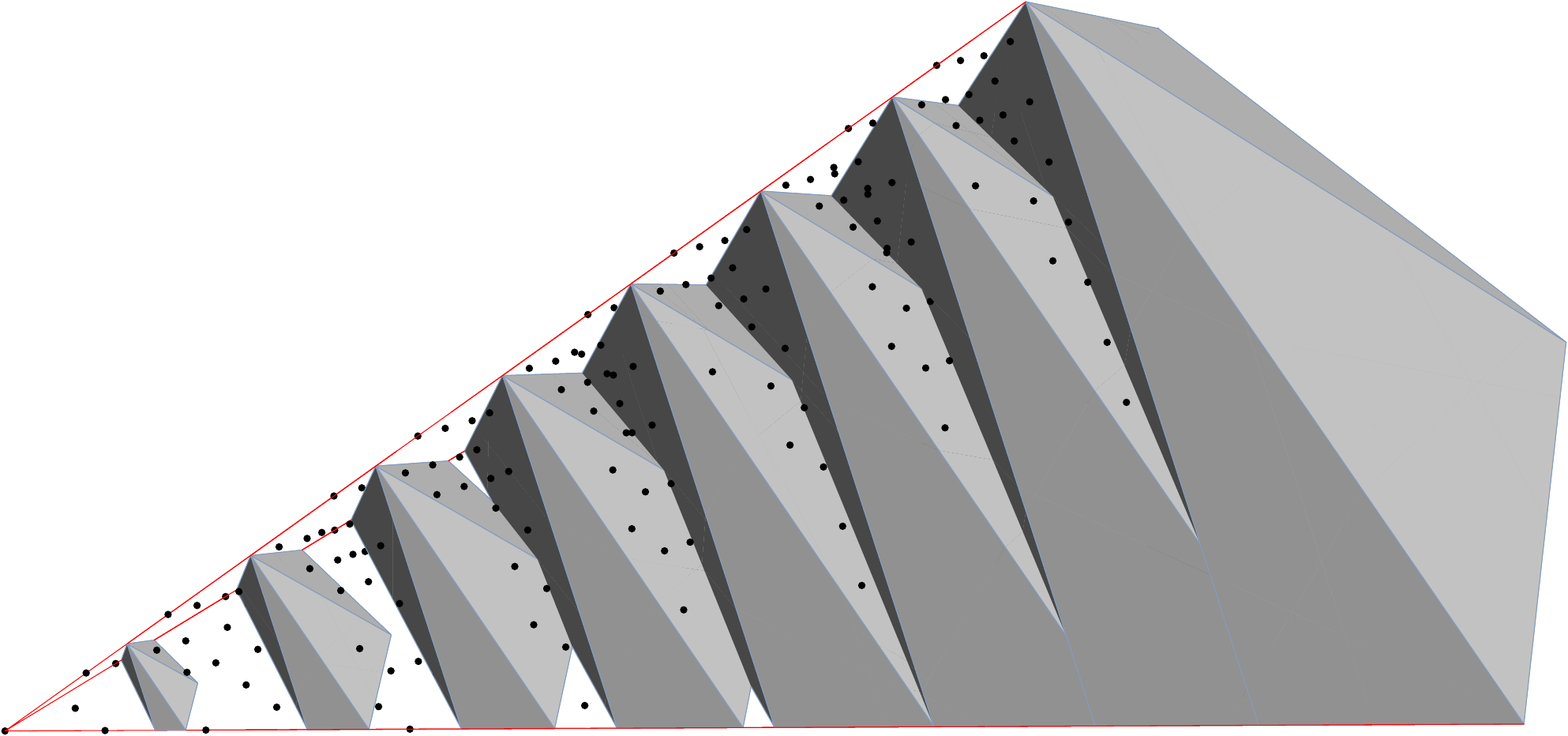}\\
\hline
\end{tabular}
\caption{Example of a non Cohen--Macaulay convex polyhedron semigroup.}\label{poliedro_1pto_2segment}
\end{center}
\end{figure}
We have
$\FraC_\FraP\setminus\FraP=\LL\cap\N^3\subset
\big(\H(\{O,\kappa_0P_1,\kappa_0P_2,\kappa_0P_3\}) \bigcup  \cup_{j\geq \kappa_0}G^j_1\big)\cap \N^3$ (see (\ref{igualdad})).

Since $G^{j+1}_1=G^j_1+P_1$ for $j\geq \kappa_0$ and the vectors $\{\overrightarrow{OP_1},\overrightarrow{OP_2},\overrightarrow{OP_3}\}$ are linearly independent, the distance from $G^j_1$ to the plane containing $\{O, P_2, P_3\}$ is strictly increasing. So, the set $(\FraC_\FraP\setminus \FraP)\cap \{P+\lambda _2 g_2+\lambda_3 g_3\mid \lambda_2,\lambda_3\in\R_\ge\}$ is finite for every $P\in \FraC_\FraP.$

Consider now any $P$ in $\FraC_\FraP\setminus \FraP,$ thus there exists $\lambda_2\in\N$ such that $P+\lambda_2 g_2\in\FraP$ but $P+(\lambda_2-1) g_2 \notin \FraP.$ Again, for a fixed $P+(\lambda_2-1) g_2,$ there exists $\lambda_3\in\N$ such that $P+(\lambda_2-1) g_2+\lambda_3g_3\in\FraP$ but $P'=P+(\lambda_2-1) g_2 +(\lambda_3-1)g_3\notin \FraP.$ Then, $P'+g_2$ and $P'+g_3$ are in $\FraP.$
By Theorem \ref{C-M}, the semigroup $\FraP$ satisfying the initial hypothesis is Cohen--Macaulay iff $\FraC_\FraP=\FraP.$
\end{proof}

Suppose $\mathcal{W}=\{P_1,P_2,P_3\},$ $\kappa_0$ is the integer defined in Lemma \ref{kcota}, and
let $Q_1,Q_1'\in G_{1}^{\kappa_0},$ $Q_2\in G_{2}^{\kappa_0}$ and $Q_3\in G_{3}^{\kappa_0}$ be four
points such that the segments $\overline{Q_1Q_2},\overline{Q_1'Q_3}$ are two edges of
$\kappa_0\CaP\cap (\kappa_0+1)\CaP$ parallel to $\overline{P_1P_2}$ and
$\overline{P_1P_3}$ respectively.

Recall that for any $j\in \N,$ the points
$\{\kappa_0P_1+jP_1,(\kappa_0+1)P_1+jP_1,\kappa_0P_2+jP_2,(\kappa_0+1)P_2+jP_2,Q_1+jP_1,Q_2+jP_2\}$ and
$\{\kappa_0P_1+jP_1,(\kappa_0+1)P_1+jP_1,\kappa_0P_3+jP_3,(\kappa_0+1)P_3+jP_3,Q'_1+jP_1,Q_3+jP_3\}$
determine the sets $\widetilde{G}_{12}^{\kappa_0+j}$ and $\widetilde{G}_{31}^{\kappa_0+j}$, respectively
(see Remark \ref{rearrange1}).
Since the segments $\overline{(Q_1+(k-\kappa_0)P_1)(Q_2+(k-\kappa_0)P_2)}$ and
$\overline{(Q_1'+(k-\kappa_0)P_1)(Q_3+(k-\kappa_0)P_3)}$ increase their lengths without limit as the
nonnegative integer $k$ increases, the distance between $G_1^k$ and $G_2^k\cup G_3^k\cup\widetilde{G}_{23}^k$ goes to infinity.
In this case, we define $k_{13}\ge \kappa_0$ to be the minimal integer such that
$P+g_2,P+g_3\notin \cup _{k\ge k_{13}} (G_{2}^k\cup G_3^k\cup \widetilde{G}_{23}^k)$ for every
$P\in \cup _{k\ge k_{13}} G_{1}^k.$
In general, for any $i\in [3],$ we define $k_{i3}\ge \kappa_0$ to be the minimal integer
such that
$P+g_{i_1}, P+g_{i_2}\notin \cup _{k\ge k_{i3}}(G_{i_1}^k \cup G_{i_2}^k \cup \widetilde{G}_{i_1i_2}^k)$ for every $P\in \cup_{k\geq k_{i3}}G_{i}^k$
with $i_1,i_2\in [3]\setminus\{i\}.$ We fix $k_3=\max \{k_{13},k_{23},k_{33}\}.$

The following lemma allows us to characterize  nontrivial cases of Cohen--Macaulay
convex polyhedron semigroups.

\begin{lemma}\label{esquinas}
Let $P_i\in\mathcal{W}$ with $i\in[3]$ and
let $P$ be a point in $\Tint(G_{i}^k)$ for any integer $k\ge k_3+h_{P_i}$ and such that $P+ g_{j}\in\FraP$ with $j\in[3]\setminus\{i\}.$ Then for fixed $t$ and $q$ the nonnegative integer quotient and least positive remainder (respectively) of the integer division $(k-k_3)/h_{P_i}$, the point $P-tg_i$ belonging to $\Tint(G_{i}^{k_3+q}) \cap \N^3$ is such that $P-tg_i+g_{j}\in \FraP.$
\end{lemma}

\begin{proof}
In order to simplify the notation, we assume $i=1,$ $j=2$ and $\mathcal{W}=\{P_1,P_2,P_3\};$ if not, some of the sets involved in (\ref{union_set}) could be empty.
Supposing that $P-tg_1+g_2\notin \FraP$, then there exists an integer $m\ge k_3$ such that $P-tg_1+g_2$ belongs to the union
\begin{equation}\label{union_set}
\overline{\LL}_m=G_{1}^{m}\cup G_{2}^{m} \cup G_{3}^{m}\cup \widetilde{G}_{12}^m\cup \widetilde{G}_{23}^m\cup \widetilde{G}_{31}^m.
\end{equation}
If $P-tg_1+g_2$ belongs to $G_1^m \cup \widetilde{G}_{12}^m\cup \widetilde{G}_{31}^m,$ by Lemmas \ref{deltaP} and \ref{prisma}, the element
$P-tg_1+g_2+tg_1$ is in $G_1^{m+th_{P_1}} \cup \widetilde{G}_{12}^{m+th_{P_1}}\cup \widetilde{G}_{31}^{m+th_{P_1}},$ and therefore
$P-tg_1+g_2+tg_1=P+g_2$ is not in $\FraP,$  which is a contradiction.
Besides, since $P-tg_1$ belongs to $\Tint(G_{1}^{k_3+q}),$ $P-tg_1+g_2$ is not in the union  $G_{2}^m \cup G_{3}^m \cup  \widetilde{G}_{23}^m.$
We conclude that $P-tg_1+g_2$ necessarily belongs to $\FraP.$
\end{proof}

Define now the set
$$\mathcal{G}=\cup _{j=0}^{h_{P_1}-1} \Tint(G_{1}^{k_3+j}) \bigcup \cup _{j=0}^{h_{P_2}-1} \Tint(G_{2}^{k_2+j}) \bigcup \cup _{j=0}^{h_{P_3}-1} \Tint(G_{3}^{k_3+j}).$$
The following result concludes  our study of the Cohen--Macaulay property.

\begin{theorem}\label{theorem_CM_3D}
Let $\CaP$ be a convex polyhedron such that
$\mathcal{W}=\{P_1,P_2,P_3\}$, or $\mathcal{W}=\{P_1,P_2\}$ and $\mathcal{W}_1=\{P_3\}$.
The affine semigroup $\FraP$ is Cohen--Macaulay if and only if
for any integer point $P$ belonging to the set $\mathcal{G}\cup \big(\H(\{O,k_3P_1,k_3P_2,k_3P_3\})\setminus \FraP\big)$ there are no two different integers $i,j\in[3]$ such that $P+g_i,P+g_j\in\FraP.$
\end{theorem}

\begin{proof}
Trivially, if $\FraP$ is Cohen--Macaulay, the theorem is satisfied.

Conversely, note that the elements in $\FraC_\FraP\setminus\FraP$ are the integer
elements in the set
$$
\begin{multlined}
T= \cup _{j\in\N} (G_{1}^{k_3+j}\cup  G_{2}^{k_3+j} \cup G_{3}^{k_3+j} \cup \widetilde{G}_{12}^{k_3+j}\cup
\widetilde{G}_{23}^{k_3+j} \cup \widetilde{G}_{31}^{k_3+j}
) \\
\cup \H(\{O,k_3P_1,k_3P_2,k_3P_3\})
\end{multlined}
$$
but not in $\FraP,$ and for any point $P$ belonging to
$\cup_{j\in\N} (\widetilde{G}_{12}^{k_3+j}\cup \widetilde{G}_{23}^{k_3+j}\cup \widetilde{G}_{31}^{k_3+j}
)\setminus \FraP,$ there exist two integers
$i_1,i_2\in [3]$ such that  $P+g_{i_1}$ and $P+g_{i_2}$ do not belong to $\FraP$ by Lemma \ref{prisma}.

Consider $P\in \FraC_\FraP\setminus \FraP,$ then $P\in (T\cap \N^3 )\setminus\FraP.$ If
$P\in \cup _{j\in\N} (G_{1}^{k_3+j}\cup  G_{2}^{k_3+j} \cup G_{3}^{k_3+j}),$ and
 $i_1,i_2\in[3]$ satisfy $P+g_{i_1},P+g_{i_2}\in\FraP,$ then there exists $P' \in \mathcal{G}$
satisfying the same property (Lemma \ref{esquinas}), which is not possible by the
hypothesis of the theorem.
Analogously, if $P\in \H(\{O,k_3P_1,k_3P_2,k_3P_3\})\setminus \FraP,$ then there do not exist
$i_1,i_2\in[3]$ with $P+g_{i_1},P+g_{i_2}\in\FraP.$

In any case, for every $P\in \FraC_\FraP\setminus \FraP,$ there are no  $i_1,i_2\in[3]$ such that $P+g_{i_1}$ and $P+g_{i_2}$ are both in $\FraP.$ By Theorem \ref{C-M}, $\FraP$ is Cohen--Macaulay.
\end{proof}

From the previous results, an algorithm to check whether a convex polyhedron semigroup is Cohen--Macaulay can be obtained. This algorithm is explicitly showed in Algorithm \ref{algorithm}. For this algorithm we assume that the semigroup associated to the initial input is a simplicial and affine semigroup.

\begin{algorithm}[H]
\caption{Sketch of the algorithm to determinate if an affine simplicial convex polyhedron semigroup is Cohen--Macaulay.}\label{algorithm}
\textbf{Input:} The convex polyhedron $\CaP.$\\
\textbf{Output:} Cohen--Macaulayness of $\FraP$ is determined.
\begin{algorithmic}[1]
\State Compute the set $\mathcal{W}.$
\If {the cardinality of $\mathcal{W}$ is less than or equal to 1}
\If {$\FraP=\FraC_\FraP$}
\State \Return $\FraP$ is Cohen--Macaulay.
\Else
\State \Return {$\FraP$ is not Cohen--Macaulay.} \EndIf
\EndIf
\State Compute $\mathcal{W}_1$ and assume that $\mathcal{W}=\{P_1,P_2,P_3\}$, or $\mathcal{W}=\{P_1,P_2\}$ and $\mathcal{W}_1=\{P_3\}$.
\State Compute $k_3$ and the minimal generators $g_1,$ $g_2$ and $g_3.$
\ForAll {$P\in \N^3\cap \left(\mathcal{G}\cup \big(\H(\{O,k_3P_1,k_3P_2,k_3P_3\})\setminus \FraP\big)\right)$}
\If{there exist $i,j\in[3]$ such that $P+g_i,P+g_j\in\FraP$}
\State \Return {$\FraP$ is not Cohen--Macaulay.}
\EndIf
\EndFor
\State \Return {$\FraP$ is Cohen--Macaulay.}
\end{algorithmic}
\end{algorithm}

The simplest case of a convex polyhedron is given by the tetrahedron.
The next result proves that if the associated semigroup is simplicial, then it is also Cohen--Macaulay.

\begin{proposition}\label{proposition_tetra_CM}
If the  semigroup associated to a tetrahedron is finitely generated and simplicial,
then it is also  Cohen--Macaulay.
\end{proposition}
\begin{proof}
Let $\CaP=\H(\{A,B,C,D\})$ be a tetrahedron satisfying the hypothesis and assuming that $A,B,C,D\in \Q^3_\geq$,
and let $\FraP$ be its associated semigroup.
We set $g_1=h_AA\in\tau_1\cap \N^3$, $g_2=h_BB\in \tau_2\cap \N^3$, and $g_3=h_CC\in \tau_3\cap \N^3$.
We consider two different settings.

First case:
assume that $O$, $C$ and $D$ belong to a rational line with
$||C||<||D||$.
In this case, we have $\LL_k=\widetilde{G}_{12}^k$ for every $k\geq \kappa_0$,
and it is straightforward to
prove that
$\LL_{i}\subset (\widetilde{G}_{12}^{\kappa_0}-(\kappa_0-i)P_1)\cap L_{Q_\geq}(\CaP)$ for
all $i\in[\kappa_0].$
By Lemma \ref{prisma}, we obtain that $P+g_1,P+g_2\notin \FraP$ for all $P\in\LL\cap\N^3$ and, therefore, by Theorem \ref{C-M}, $\FraP$ is Cohen--Macaulay.

Second case: assume that $\tau_1\cap\CaP=\{A\},$ $\tau_2\cap\CaP=\{B\},$ and
$\tau_3\cap\CaP=\{C\}.$ Let $Q$ be the point obtained from the intersection of
the line containing the origin and $D$ with the triangle $\H(\{A,B,C\}).$ We divide  $\CaP$ in three tetrahedrons:  $\CaP_1=\H(\{B,C,Q,D\})$, $\CaP_2=\H(\{A,C,Q,D\})$, and $\CaP_3=\H(\{A,B,Q,D\})$.
Note that $\CaP_1$, $\CaP_2$, $\CaP_3$ satisfy the previous case, and, therefore,  for every $P\not\in \FraP$ we have $P+g_i,P+g_j\not \in \FraP_{k}\subset \FraP$ with $i,j\in [3]\setminus \{k\}$ for some $k\in[3]$.
By Theorem \ref{C-M}, $\FraP$ is Cohen--Macaulay.
\end{proof}

In the following example we give a Cohen--Macaulay convex polyhedron semigroup associated to a polyhedron that is not a tetrahedron.

\begin{example}
Let $\CaP$ be the non-tetrahedron polyhedron obtained from the convex hull of the vertex set
\[\{ (3, 3, 2), (2, 3, 1), (1, 2, 3), (3/2, 3, 9/2), (33/16, 27/8, 63/16)\}.\]
The associated convex polyhedron semigroup
is minimally generated by 6 elements. They can be computed as follows:
\begin{verbatim}
In[1]: v = [[3, 3, 2], [2, 3, 1], [1, 2, 3], [3/2, 3, 9/2],
                [33/16, 27/8, 63/16]];
In[2]: msgCHSgr(v)
Out[2]: [[1, 2, 3], [2, 3, 1], [2, 3, 2], [2, 3, 3],
                [3, 3, 2], [4, 6, 7]]
\end{verbatim}

The command {\tt isCohenMacaulay} of \cite{PROGRAMA} checks whether a simplical convex polyhedron semigroup is Cohen--Macaulay,
\begin{verbatim}
In[3]: isCohenMacaulay(v)
Out[3]: True
\end{verbatim}
In this case, the semigroup is Cohen--Macaulay. The same result is obtained using
{\tt Macaulay2} (see \cite{Macaulay2}).
\end{example}

\section{Computing a family of Gorenstein affine simplicial convex polyhedron semigroups}\label{gorenstein}

In this section, we obtain a family of Gorenstein semigroups from their associated  convex polyhedra. As in the previous sections, we assume that the semigroup $S=\langle g_1,\ldots g_t \rangle\subset \N^n$ is simplicial and $L_{\R_\ge }(S)$ is generated by $\{g_1,\ldots g_n\}.$ The set $\ap(g)=\{x\in S\mid x-g\notin S\}$ is called the Ap\'{e}ry set of $g\in S.$
In \cite[Section 4]{RosalesCM}, the following result, which characterizes  Gorenstein semigroups, is  given.

\begin{theorem}\label{Gorenstein_rosales}
For a simplicial semigroup $S,$ the following conditions are equivalent:
\begin{enumerate}
\item $S$ is Gorenstein.
\item $S$ is Cohen--Macaulay and $\cap_{i\in [n]} \ap(g_i)$ has a unique maximal element (with respect to the order defined by $S$).
\end{enumerate}
\end{theorem}

From this, we can obtain Gorenstein semigroups by looking at the intersection
$\cap_{i\in [n]} \ap(g_i).$
Consider the convex polyhedron $\CaP$ (tetrahedron) equal to the convex hull of
\begin{equation}\label{f1}
\{(4,0,0),(4+2k,0,0),(4+k,k,0),(4+k,0,1)\},
\end{equation}
with $k$ an integer greater than or equal to 2. Denote by
$\FraP$ the affine semigroup associated to $\CaP$.
We have that $\FraP$ is simplicial and that $g_1=(4,0,0)$, $g_2=(4+k,k,0),$ and $g_3=(4+k,0,1)$ are a set of generators of
the cone associated to $\FraP.$ Note that,  by
Proposition \ref{proposition_tetra_CM},  this affine semigroup is Cohen--Macaulay.

\begin{lemma}\label{ap_n3}
Let $\FraP$  be the tetrahedron with the set of vertices (\ref{f1}).
For every $P=(p_1,p_2,p_3)\in\FraP$ with $p_3>0,$ the element $P-g_3$ belongs to $\FraP.$
\end{lemma}

\begin{proof}
Given $P\in \FraP,$ there exists $i\in \N$ such that $P\in i\CaP.$ Assume that $i>1,$ if not, $P=(4+k,0,1)$ and $P-g_3=(0,0,0)\in\FraP.$ So, we can write $P$ as
$$P=\lambda _1 (4i,0,0)+\lambda_2 ((4+2k)i,0,0)+\lambda_3((4+k)i,ki,0)+\lambda_4((4+k)i,0,i)$$
where $\sum_{j\in[4]}\lambda_j=1$ and $\lambda_1,\ldots ,\lambda_4\in [0,1].$ Note that $p_3\le i,$ and then $i=p_3+t$ with $t\in \N.$ We also assume that $p_3\in[i-1],$ otherwise, $P=i(4+k,0,1)$ and $P-g_3=(i-1)(4+k,0,1)\in\FraP.$ Since $P=(p_1,p_2,p_3),$ the element $P$ belongs to $i\CaP\cap \{z=p_3\}.$ After some basic computations the reader can check that this set is the convex polygon
$$\H\big(p_3(4+k,0,1)+ t\{(4,0,0),(4+2k,0,0),(4+k,k,0) \} \big).
$$
Hence, there exist $\nu_1,\nu_2,\nu_3\in [0,1]$ with $\nu_1+\nu_2+\nu_3=1$ satisfying
$$
\begin{multlined}
P= \nu_1 \big(t(4,0,0)+p_3(4+k,0,1)\big) + \nu_2 \big(t(4+2k,0,0)+p_3(4+k,0,1)\big)\\
+\nu_3 \big(t(4+k,k,0)+p_3(4+k,0,1)\big)\\
= \underbrace{t[\nu_1(4,0,0)+\nu_2 (4+2k,0,0) +\nu_3  (4+k,k,0)]}_{\in t\CaP\cap\N^3}+\underbrace{p_3(4+k,0,1)}_{\in p_3\CaP\cap\N^3}.
\end{multlined}
$$
We can conclude that $P-g_3\in\FraP.$
\end{proof}

In \cite{GV_CM_Go}, it is proved that the intersection $\ap (g_1) \cap \ap (g_2) \cap \{z=0\}$ is the set shown in Table \ref{interseccion_apery}.
Besides, in \cite{GV_CM_Go} it is also proved that $(10+k,k-1,0)$ is the unique maximal element of $\ap (g_1) \cap \ap (g_2) \cap \{z=0\}$, and
therefore $\FraP$ is Gorenstein by Theorem \ref{Gorenstein_rosales}.
{\small
	\begin{table}[h]
		\begin{center}
			\begin{tabular}{|ccc|}
				\hline
				$\ap(g_1)\cap \ap(g_2)\cap \{y=0,z=0\}$ & = & $\{
				(0,0,0),(5,0,0),(6,0,0), (7,0,0)\}$ \\ \hline
				$\ap(g_1)\cap \ap(g_2)\cap \{{y=1,z=0\}}$ & = & $\{(5,1,0),(6,1,0),(7,1,0),(8,1,0)\}$ \\ \hline
				\vdots & \vdots & \vdots \\ \hline
				\multirow{2}{*}{$\ap(g_1)\cap \ap(g_2)\cap \{y=k-2,z=0\}$}& = & $\{(2+k,k-2,0),(3+k,k-2,0) ,$\\
				& & $(4+k,k-2,0) , (5+k,k-2,0) \}$\\ \hline
				\multirow{2}{*}{$\ap(g_1)\cap \ap(g_2)\cap \{y=k-1,z=0\}$} & = & $\{(3+k,k-1,0),(4+k,k-1,0),$\\
				& & $(5+k,k-1,0),(10+k,k-1,0) \}$\\ \hline
				$\ap(g_1)\cap \ap(g_2)\cap \{y\ge k,z=0\}$ & = & $\emptyset$\\ \hline
			\end{tabular}
			\caption{$\ap(g_1)\cap \ap(g_2)\cap \{z=0\}$.}\label{interseccion_apery}
		\end{center}
\end{table}}

\begin{corollary}\label{Go_thetrahedrom}
For every $k\in\N\setminus\{0,1\}$
the affine semigroup $\FraP$ associated to the convex hull of the vertex set $\{(4,0,0),(4+2k,0,0),(4+k,k,0),(4+k,0,1)\}$ is Gorenstein.
\end{corollary}

\begin{example}
Let $\CaP$ be the convex hull of $\{(4,0,0),(7,3,0),(10,0,0),(7,0,1)\}$ and $\FraP$ be its associated affine convex polyhedron semigroup.
It is straightforward to check that $\FraP$ is simplicial and that $\CaP$ is a
tetrahedron. So, $\FraP$ is Cohen--Macaulay, and,
by Corollary \ref{Go_thetrahedrom}, it is also Gorenstein.
The minimal system of generators of $\FraP$ can be computed with the command
{\tt msgCHSgr} of \cite{PROGRAMA}.
With this command, we obtain
$$
\begin{multlined}
\{(4,0,0),(7,3,0),(7,0,1),(6,0,0),(7,0,0),(5,0,0),\\(6,1,0),
 (8,1,0),(7,1,0),(5,1,0),(6,2,0),(8,2,0),(7,2,0)\}.
\end{multlined}
$$
Using the program {\tt Macaulay2} (see \cite{Macaulay2}) we also obtain that $\FraP$ is
Gorenstein.
\end{example}

\section{Buchsbaum affine simplicial convex polyhedron semigroups}\label{buchsbaum}
In this section, we characterize the Buchsbaum affine simplicial convex polyhedron semigroups. So, we assume the semigroups appearing in this section are
simplicial and finitely generated.

Given $S\subset \N^n$, a simplicial semigroup minimally generated by $\{g_1,\ldots ,g_t\},$ we write
$\overline{S}$ for the semigroup $\{a\in\N^{n}\mid a+g_i\in S,\, \forall i=1,\ldots ,t\}$. Trivially, $S\subset \overline{S}\subset \FraC_S.$ The following result gives a characterization of Buchsbaum rings in terms of their associated semigroups, when the semigroup is simplicial.

\begin{theorem}\cite[Theorem 5]{RosalesBuchs}\label{RosalesBuchs}
The following conditions are equivalent:
\begin{enumerate}
\item $S$  is Buchsbaum.
\item $\overline{S}$ is Cohen--Macaulay.
\end{enumerate}
\end{theorem}

To check whether a simplicial  convex polyhedron semigroup
$\FraP\subset \N^3$ is Buchsbaum, we study the Cohen--Macaulayness of $\overline{\FraP}$.
Note that if for any $i\in [3]$  the set $\tau_i\cap \CaP$ has
only one point, then $\tau_i\cap \FraP= \tau_i\cap \overline{\FraP}$ and it is generated by one element. We assume that this element is $g_i.$
Otherwise, $\tau_i\cap \CaP$ is a segment, and we can obtain the generators
of $\tau_i\cap \overline{\FraP}$ by checking the points in the finite set
$\tau_i\setminus \FraP.$

As in Section \ref{sectionCM}, we consider different cases, depending on the intersections of $\tau_i\cap \CaP$ for $i\in [3].$
The following result characterize the easiest cases.

\begin{lemma}
Let $\CaP\subset \R^3_\ge$ be a convex polyhedron such that
$\mathcal{W}=\emptyset,$ or $\mathcal{W}=\{P_1\}$ and $\mathcal{W}_1=\{P_2,P_3\}$.
Then
$\FraP$ is Buchsbaum if and only if $\overline{\FraP}=\FraC_\FraP.$
\end{lemma}

\begin{proof}
By Theorem \ref{RosalesBuchs}, $\FraP$ is Buchsbaum if and only if $\overline{\FraP}$ is Cohen--Macaulay. If we consider $\mathcal{W}=\emptyset,$ by using a similar reasoning to the proof of Lemma \ref{1pto_2segmentosCM}
, it can be shown that $\overline{\FraP}$ is Cohen--Macaulay if and only if $\overline{\FraP}=\FraC_\FraP.$

For the case $\mathcal{W}=\{P_1\}$ and $\mathcal{W}_1=\{P_2,P_3\}$, note that the intersection
$\Tint (\cup _{j\in \N}G_{1}^{\kappa_0+j})\cap \overline{\FraP}$ is empty. Thus,
$\overline{\FraP}\setminus\FraP$ is a finite set included in
$\H(\{O,\kappa_0P_1,\kappa_0P_2,\kappa_0P_3\}).$ So, the lemma can be proved analogously to Lemma
\ref{1pto_2segmentosCM}.
\end{proof}

As in Theorem \ref{theorem_CM_3D}, we focus on the cases:
\begin{itemize}
\item
$\mathcal{W}=\{P_1,P_2\}$ and $\mathcal{W}_1=\{P_3\}$.
\item
$\mathcal{W}=\{P_1,P_2,P_3\}$.
\end{itemize}
Note that for all integers $j\ge \kappa_0$ and for all points $P$ belonging to $$(G_{1}^j\cup  G_{2}^j \cup G_{3}^j\cup \widetilde{G}_{12}^j\cup \widetilde{G}_{23}^j
\cup \widetilde{G}_{31}^j)\setminus\FraP,$$ $P\notin \overline{\FraP}$ (see Lemma \ref{deltaP} and Lemma \ref{prisma}). So, $\overline{\FraP}\setminus\FraP$ is a finite set included in $\H(\{O,\kappa_0P_1,\kappa_0P_2,\kappa_0P_3\}).$

\begin{theorem}\label{theorem_Buch_3D}
With the same assumptions as in Theorem \ref{theorem_CM_3D},
$\FraP$ is Buchsbaum if and only if
for any integer point $P$ belonging to the set $\big(\mathcal{G}\cup \H(\{O,k_3P_1,k_3P_2,k_3P_3\})\big)\setminus \overline{\FraP}$ there are no two different integers $i,j\in[3]$ such that $P+g_i,P+g_j\in\overline{\FraP}.$
\end{theorem}

\begin{proof}
The proof is analogous to that of Theorem \ref{theorem_CM_3D}, but applied to $\overline{\FraP}$ instead of $\FraP.$
\end{proof}

\begin{corollary}
The affine simplicial convex polyhedron semigroup associated to any tetrahedron is Buchsbaum.
\end{corollary}

\begin{proof}
For such a convex polyhedron, the semigroups $\FraP$ and $\overline{\FraP}$ are equal, and $\FraP$ is Cohen--Macaulay by Proposition \ref{proposition_tetra_CM}.
\end{proof}

The following example provides us with a Buchsbaum simplicial affine
semigroup that it is not Cohen--Macaulay.

\begin{example}
Let $\CaP$ be the polyhedron
$$
\begin{multlined}
\H(\{(24/5, 12/5, 12/5), (8/3, 16/3, 8/3), (8/3, 8/3, 16/3), \\
(152/33, 152/33, 16/3), (152/33, 16/3, 152/33), (856/165, 68/15, 68/15)\}).
\end{multlined}$$
It is easy to check that $\CaP$ is not a tetrahedron.
Using the command {\tt msgCHSgr} of  \cite{PROGRAMA} we obtain that
its associated convex polyhedron semigroup $\FraP$
is minimally generated by the following of $71$ elements:
{\small
$$
\begin{array}{l}
\{(3, 3, 5), (3, 4, 4), (3, 5, 3), (4, 3, 3),
(4, 3, 4), (4, 4, 3), (4, 4, 4), (4, 4, 5),\\
(4, 5, 4), (5, 4, 4), (6, 6, 9), (6, 7, 8),
(6, 8, 7), (6, 9, 6), (8, 5, 7), (8, 7, 5),\\
(8, 8, 16), (8, 9, 15),(8, 10, 14), (8, 11, 13),
(8, 12, 12), (8, 13, 11), (8, 14, 10),\\
(8, 15, 9),(8, 16, 8), (9, 6, 6), (9, 6, 7), (9, 7, 6),
(9, 9, 9), (9, 9,   10),(9, 9, 16),\\
(9, 10, 9), (9, 10, 15), (9, 11, 14), (9, 12, 13), (9, 13, 12),
(9, 14, 11),(9, 15, 10), \\
(9, 16, 9), (10, 7, 7),(10, 8, 13), (10, 9, 9),(10, 10, 16),
(10, 11, 15),(10, 12, 14), \\
(10, 13, 8), (10, 13, 13), (10, 14, 12),(10, 15, 11),
(10, 16,10), (11, 11, 16), (11, 12, 15),\\
(11, 13, 14), (11, 14, 13), (11, 15, 12), (11, 16, 11),
(12, 12, 16), (12, 13, 15), (12, 14, 14), \\
(12, 15, 13),(12, 16, 12), (13, 8, 9), (13, 9, 8), (14, 8, 8),
(14, 9, 9), (18, 10, 11), \\
(18, 11, 10), (19, 10, 10),
(19, 11, 11), (24, 12, 12), (24, 13, 13)\}.
\end{array}
$$}

\noindent
Furthermore, the semigroup $\overline{\FraP}$ is equal to the convex polyhedron
semigroup associated to the tetrahedron with vertex set
$\{(24/5, 12/5, 12/5), (8/3, 16/3, 8/3), (8/3, 8/3, 16/3), (16/3, 16/3, 16/3)\}.$
Thus, $\overline \FraP$ is Cohen--Macaulay, and therefore $\FraP$ is Buchsbaum.

By using the command {\tt isBuchsbaum} of \cite{PROGRAMA}, the time required for checking that  $\FraP$ is Buchsbaum
was around a minute. In contrast, {\tt Macaulay2} (see \cite{Macaulay2})
took more than two hours.
\end{example}

\subsection*{Acknowledgement}
This publication and research was partially supported by grants by INDESS (Research University Institute for Sustainable Social Development), Universidad de C\'{a}diz, Spain.

\end{document}